\documentclass{article}

\usepackage{arxiv}

\usepackage[T1]{fontenc}    
\usepackage[numbers]{natbib}
\usepackage{hyperref}       
\usepackage{url}            
\usepackage{booktabs}       
\usepackage{amsfonts, amsmath, amssymb}       
\usepackage{nicefrac}       
\usepackage{microtype}      
\usepackage{graphicx}

\usepackage{algorithm}
\usepackage{algpseudocode}
\usepackage{tikz}
\usepackage{standalone}
\usetikzlibrary{decorations.pathreplacing}
\usepackage{pgfplots}
\usepgfplotslibrary{dateplot}
\pgfplotsset{compat=1.18}
\usepackage{amsthm}

\newcommand{\mub}{\boldsymbol{\mu}}

\newcommand{\lambdab}{\boldsymbol{\lambda}}
\newcommand{\thetab}{\boldsymbol{\theta}}
\newcommand{\gammab}{\boldsymbol{\gamma}}
\newcommand{\varthetab}{\boldsymbol{\vartheta}}
\newcommand{\kappab}{\boldsymbol{\kappa}}
\newcommand{\R}{\mathbb{R}}
\newcommand{\Z}{\mathbb{Z}}
\newcommand{\E}{\mathbb{E}}

\newcommand{\alphab}{\boldsymbol{\alpha}}

\newtheorem{theorem}{Theorem}

\newtheorem{assumption}{Assumption}

\title{On the singularity of the Fisher Information matrix in the sine-skewed family on the $d$-dimensional torus}

\author{Emily Schutte \\
	University of Amsterdam \\ 
    The Netherlands\\
	\texttt{e.j.schutte@uva.nl} \\
	\And
	Sophia Loizidou \\
    University of Luxembourg\\
    Luxembourg \\
	\texttt{sophia.loizidou@uni.lu} \\
    \And
	Vincent Laheurte \\
    University of Luxembourg\\
    Luxembourg \\
	\texttt{vincent.laheurte@uni.lu} \\
}

\date{}

\renewcommand{\shorttitle}{On the singularity of the Fisher Information matrix in the sine-skewed family}

\begin{document}
\maketitle

\begin{abstract}
Skewed distributions are fundamental in modelling asymmetric data on the $d$-dimensional torus. 
In this context, asymmetry is introduced through the sine-skewing mechanism, {which is the only skewing mechanism that has been proposed on the hyper-torus in the literature}.
Some sine-skewed models are known to suffer from  a singular Fisher information matrix in the vicinity of symmetry, which poses a significant issue for inferential purposes. 
It is an open question  to determine for  which sine-skewed models Fisher information singularity occurs. 
In this paper, a general characterization of the class of models that exhibit this singularity is given in the general $d$-dimensional setting.
\end{abstract}

\keywords{Directional statistics \and Asymmetry \and Fisher information singularity}

\title{On the singularity of the Fisher Information matrix in the sine-skewed family on the $d$-dimensional torus} 

\section{Introduction}\label{sec: intro}

Several real-world datasets can naturally be represented as data on the $d$-dimensional torus.
Notable examples arise from 
the area of bioinformatics, particularly in the context of the protein folding prediction problem, where angles play a central role \citep{kato2025trivariatewrappedcauchycopula, mardia_multivariate_2008, mardia_protein_2007}. 
Other examples are for instance RNA data \citep{nodehi_estimation_2018} 
and the circadian clock of two different tissues of a mouse \citep{liu_phase_2006}.
Beyond the bioinformatics world, toroidal data arise in various domains: 
the direction of animal movement \citep{mastrantonio_modelling_2022}, 
morphological data from human neurons \citep{puerta_regularized_2015},
wind direction data
\citep{kato_distribution_2009} and 
directions of steepest descent before and after an earthquake \citep{rivest_decentred_1997}, to cite but these.

Many distributions have been proposed in the literature to model $d$-dimensional angular data.
On the circle ($d=1$), key distributions are the von Mises, the cardioid, the wrapped Cauchy and the wrapped Normal distributions, to name a few.
On the two-dimensional torus,  commonly used models are the bivariate wrapped Cauchy distribution \citep{kato_mobius_2015, kato_distribution_2009}, 
the bivariate von Mises distribution \citep{mardia_statistics_1975} and its submodels, the Sine \citep{singh_probabilistic_2002} and Cosine \citep{mardia_protein_2007} distributions. 
Considering higher dimensions, the trivariate wrapped Cauchy copula \citep{kato2025trivariatewrappedcauchycopula} can model three angles, while the multivariate wrapped normal \citep{Baba1981}
and the multivariate extensions of the Sine and Cosine distributions \citep{mardia_multivariate_2008} can be used for any dimension. For a general overview, we refer the reader to \cite{ley_modern_2017} and \cite{pewsey_recent_2021}.

A strong limitation of the above-mentioned models arises from their symmetric nature, which is not able to capture the asymmetric phenomena that can be observed in circular and toroidal data \citep{ameijeiras2022flexible, rad2022enhancing, HARVEY2024105450}.
Recognizing this  need on the circle, \cite{Umbach2009} and 
\cite{abe2011symmetric} took inspiration from the popular skew-normal distribution of \cite{Azzalini1985}
and  proposed  the  sine-skewed distributions on the circle. \cite{ameijeiras-alonso_sine-skewed_2022} extended this proposal to the $d$-dimensional torus, leading to a density of the form
\begin{align}\label{eq: sine_skewed_density}
    \thetab\mapsto f_{\mub, \lambdab}(\thetab; \varthetab) := f_0(\thetab - \mub; \varthetab)\left(1+\sum_{j=1}^{d} \lambda_j \sin (\theta_j-\mu_j)\right) 
\end{align}
subject to
$\sum_{j=1}^{d} |\lambda_j| \leq 1,$ where $f_0$ is a symmetric base density on the $d$-dimensional torus (like all those mentioned above), $\mub = (\mu_1 ,\ldots, \mu_d)^\intercal \in [-\pi, \pi)^d$ is the location parameter,   $\varthetab$ stands for any non-location parameters and $\lambdab  = (\lambda_1 ,\ldots, \lambda_d)^\intercal \in (-1,1)^d$ plays the role of the skewness parameter. 
For the rest of the paper, we omit $\varthetab$ from the notation for the sake of readability.

These skewing models respond to the need of asymmetric distributions and have increased the flexibility of distributions on the $d$-dimensional torus, yet this comes at a cost. Indeed, a major problem arises from the Fisher information matrix (FIM) of these models and the ensuing inferential challenges. On the real line, it has been observed by \cite{Azzalini1985} that the FIM for the skew-normal distribution becomes singular in the vicinity of symmetry, that is, when $\lambda = 0$.
A lot of literature has been dedicated to this issue, see in particular \cite{Hallin_singularity2012} for a thorough account.  
Unfortunately, this singularity also arises in the sine-skewed family of distributions as defined in \eqref{eq: sine_skewed_density}, which may come as no surprise since their construction follows the same principle as the skew-normal distribution of \cite{Azzalini1985}. 
\cite{ley_simple_2014} showed that, in the circular case, the FIM is singular in the vicinity of symmetry if and only if the sine-skewing transformation is applied to the von Mises distribution, hence no other symmetric density yields such a result. This fact is due to an unfortunate  collinearity between location and skewness scores.
 \cite{symmetry_test} note that in the two-dimensional case singularity is observed for the sine-skewed Cosine distribution.
This observation is to be expected, as the Cosine distribution is an extension of the von Mises model on the torus.
However, interestingly, the Sine distribution, which is a different two-dimensional extension of the same model, does not exhibit the same issue. 
This contrast highlights that extending the singularity results beyond the circle is a non-trivial task.

Such singularities are problematic because singularity of the FIM indicates lack of information for the parameters to be uniquely identifiable from the data. 
A major inferential consequence is that the asymptotic normality of the maximum likelihood estimator fails to hold. 
Moreover, when the FIM is non-singular the rate of convergence to the true parameter is $O(n^{-\frac{1}{2}})$.  
However, when the matrix is singular this rate can be significantly slower and the likelihood can also exhibit irregularities, such as bimodality \citep{Rotnitzky2000}.
This has severe implications on many statistical procedures that rely on the asymptotic normality property, such as hypothesis testing and construction of confidence intervals.
Singularity of the FIM is also closely related to the smeariness of (nonparametric) sample Fr\'echet means \citep{smeariness1}. 
Theorem 3 of \cite{smeariness2} shows that, under the assumption that the underlying density differs from the uniform distribution at $-\pi$ for the first time in its $k^\text{th}$ derivative there, the convergence rate of the intrinsic sample mean is $n^{-\frac{1}{2(k+1)}}$, which is slower than the previously mentioned $n^{-\frac{1}{2}}$ rate.

Given the lack of knowledge as to which sine-skewed distributions on the $d$-dimensional torus suffer from a FIM singularity, we fully answer this open question in this paper. 
In Section~\ref{sec: main_theorem} we provide a general form  of symmetric density functions on the hyper-torus that is satisfied if and only if
their sine-skewed version  admits a singular FIM in the vicinity of symmetry.
In Section~\ref{sec: examples} we consider some well-known distributions and study whether they satisfy the proposed characterization or not, hence suffer from FIM singularity. 
Lastly, we conclude the paper with a discussion in Section~\ref{sec: conclusion}.


\section{Fisher information  singularity for the $d$-dimensional torus} \label{sec: main_theorem}

In this section we establish the main result of our paper.
Considering model \eqref{eq: sine_skewed_density}, we require $f_0$ to be a component-wise $2\pi$-periodic, symmetric density around $\mub \in [-\pi, \pi)^d$. 
More formally, we require $f_0 \in \mathcal{F}$, where $\mathcal{F}$ is defined as
\begin{align*}\label{eq: def: F}
    \mathcal{F} := \left\{\begin{array}{ll}
& f_0(\thetab) > 0 \,\mbox{a.e.}\, \ \forall \thetab \in [-\pi, \pi)^d, \\
f_0(\thetab): & f_0(\thetab + 2\pi \boldsymbol{k}) = f_0(\thetab)\, \ \forall \boldsymbol{k} \in\Z^d,\\
& f_0(\mub + \thetab) = f_0(\mub -\thetab)\, \ \forall \thetab \in [-\pi, \pi)^d, \\
& \int_{[-\pi,\pi)^d} f_0(\thetab) {\mathrm d} \thetab = 1.
\end{array} \right\}.
\end{align*}
We also require the following mild assumption.
\begin{assumption}\label{assumption}
    The mapping $\thetab \mapsto f_0(\thetab)$ is differentiable on $[-\pi, \pi)^d$.
\end{assumption}
In Theorem~\ref{thm: result} we derive a general expression for symmetric densities which is valid if and only if the FIM of their resulting sine-skewed version is singular in the vicinity of symmetry.
The proof of the theorem is based on the fact that the FIM is singular if and only if the elements of the score function for location and skewness parameters are linearly dependent. 
Immediate calculations yield that the score function of \eqref{eq: sine_skewed_density} in the vicinity of symmetry corresponds to
\begin{equation} \label{eq: score}
    S_{f_0} 
    =  \left( 
         -\frac{\frac{\partial}{\partial \theta_1}f_0(\thetab-\mub)}{f_0(\thetab-\mub)},
        \dots,
        -\frac{\frac{\partial}{\partial \theta_d}f_0(\thetab-\mub)}{f_0(\thetab-\mub)},
        \sin(\theta_{1}-\mu_1),
        \dots,
        \sin(\theta_{d}-\mu_d) 
    \right)^\intercal.
\end{equation}
For notational convenience we focus on the partial derivative with respect to $\thetab$, using the fact that for $i \in \{1,\ldots,d\}$, $-\frac{\partial}{\partial \theta_i} \log f_0(\thetab - \mub) =\frac{\partial}{\partial \mu_i}\log f_0 (\thetab - \mub)$.
For the sake of illustration we calculate \eqref{eq: score} for $f_0(\thetab)$ being the Cosine density and
we drop $\mub$ from the notation for convenience.
The density is given by
\begin{equation}\label{eq: cosine_density}
        f_C(\thetab) = C_c \exp\left\{ \kappa_1 \cos(\theta_1) + \kappa_2 \cos(\theta_2) + \beta \cos(\theta_1 -\theta_2) \right\}
\end{equation}
    where $C_c$ is the appropriate normalizing constant and $\kappa_1, \kappa_2\geq0, \beta\in\R$ are the concentration parameters.
    It can be calculated that
    \begin{equation*}
        S_{f_{C}} = \left( \kappa_1 \sin(\theta_1) + \beta\sin(\theta_1-\theta_2), \kappa_2 \sin(\theta_2) - \beta\sin(\theta_1-\theta_2), \sin(\theta_1), \sin(\theta_2) \right)^\intercal
    \end{equation*}
    which has linearly dependent elements {as $\left\langle (1, 1, -\kappa_1, -\kappa_2)^\intercal, S_{f_{C}} \right\rangle = 0 $,
    where $\langle \cdot \, , \cdot \rangle$ denotes the inner product of vectors.}
    This confirms that the Cosine model suffers from a singular FIM in the vicinity of symmetry, as stated in \cite{symmetry_test}.

    Before providing the statement of the theorem we define the function
    \begin{equation} \label{eq: solution}
    h_0( \thetab {- \mub}) := f_0(\thetab {- \mub}) \exp \left( \sum^d_{i=1}\gamma_i \cos(\theta_i {- \mu_i}) \right)
    \end{equation}
    for $\gammab = (\gamma_1, \dots,\gamma_d)^\intercal \in \R^d$ and $\mub = (\mu_1 ,\ldots, \mu_d)^\intercal \in [-\pi, \pi)^d$.
    The proof of the theorem can be found in the Supplementary Material.

\begin{theorem}\label{thm: result}
Let $f_0\in\mathcal{F}$ satisfy Assumption~\ref{assumption}.
The FIM of the sine-skewed version of $f_0(\thetab {- \mub})$ in the vicinity of symmetry is singular if and only if there exists a vector $\alphab = \left( \alpha_1, \ldots, \alpha_d \right)^\intercal\in\R^d$ with $\alpha_i \neq 0$ for $i\in\{1,\ldots,d\}$ 
such that $h_0(\thetab {- \mub})$ defined in \eqref{eq: solution} 
satisfies
\begin{equation}\label{eq: condition}
    h_0(\thetab {- \mub} + t\alphab) = h_0(\thetab {- \mub})
\end{equation}
for all $t\in \mathbb{R}$ and $\thetab \in [-\pi,\pi)^d $.
\end{theorem}




{In other words, the FIM of the sine-skewed version of $f_0(\thetab)$ will be singular in the vicinity of symmetry if and only if it can be written as
\begin{equation*}
    f_0( \thetab - \mub) = h_0(\thetab - \mub) \exp \left( -\sum^d_{i=1}\gamma_i \cos(\theta_i - \mu_i) \right)
\end{equation*}
for $h_0(\thetab)$ satisfying \eqref{eq: condition}.}
This result fully characterizes which symmetric densities $f_0$, combined with the sine-skewing mechanism, suffer from FIM singularities in the vicinity of symmetry and which not, hence completely solves this open question. 
In the next section, we apply it on several well-known distributions from the literature in order to provide an overview of occurring FIM singularities. 
This result is very important as it informs researchers  about when inferential procedures work as usual and when problems occur, as explained in the Introduction.

The result of Theorem~\ref{thm: result} is in fact more generally valid as it holds for any skewing mechanism that has score vector \eqref{eq: score} in the vicinity of symmetry. 
As stated in \cite{symmetry_test}, a skewing mechanism of the form 
\begin{align*}\label{eq: sine_skewed_density_new_proposal}
    \thetab\mapsto  f_0(\thetab {- \mub}) \prod_{j=1}^{d} \left(1 + \lambda_j \sin (\theta_j {- \mu_j})\right),
\end{align*}
subject to $|\lambda_j| \leq 1, j = 1,\ldots,d$, has the same score function, and thus suffers from singularity issues.

A different skewing mechanism was proposed by \cite{Bekker2022} on the 2-dimensional torus.
The proposed model is of the form
\begin{equation*}
    \thetab\mapsto  C_{m, f_0} f_0(\theta_1 {- \mu_1}, \theta_2 {- \mu_2})\left(\frac{1 + \lambda_1 \sin (\theta_1 {- \mu_1}) + \lambda_2 \sin (\theta_2 {- \mu_2})}{2}\right) ^m
\end{equation*}
for a normalizing constant $C_{m, f_0}$, a positive integer $m$ {and location parameter $\mub = (\mu_1 , \mu_2)^\intercal \in [-\pi, \pi)^2$}.
The score function of this skewing mechanism, in the vicinity of symmetry, is given by
\begin{equation*}
    \left(
        -\frac{\frac{\partial}{\partial \theta_1}f_0(\thetab {- \mub})}{f_0(\thetab {- \mub})},
        -\frac{\frac{\partial}{\partial \theta_2}f_0(\thetab {- \mub})}{f_0(\thetab {- \mub})},
        m\sin(\theta_{1} {- \mu_1}),
        m\sin(\theta_{2} {- \mu_2}) 
        \right)^\intercal.
\end{equation*}
Thus, the result of Theorem~\ref{thm: result} applies as well to this skewing mechanism, by rescaling $\gamma_i$ by $\gamma_i/m$.

\section{Examples of known distributions} \label{sec: examples}
In this section, we investigate which distributions {can be written in the form} \eqref{eq: solution} for a function $h_0(\thetab)$ that satisfies \eqref{eq: condition}. 
In other words, we will go through some of the best-known symmetric models from the literature and discuss which ones are prone to FIM singularity under sine-skewing and which not.
As mentioned in the Introduction, the von Mises distribution on the circle admits a singular FIM, hence we focus on its extensions. 
{For simplicity, we drop the location parameter $\mub$ from the notation.}
\begin{itemize}
    \item \textbf{product of independent von Mises distributions:} The density is given by
    \begin{equation*}
        f_{ivM}(\thetab) = \left(\prod_{i=1}^d \frac{1}{ 2\pi I_0(\kappa_i)}\right) \exp \left( \sum_{i=1}^d \kappa_i \cos(\theta_i) \right)
    \end{equation*}
    with concentration parameters $\kappa_i$ for $i=1,\ldots,d$.
    Setting $\gamma_i = -\kappa_i$, it holds that
    \begin{align*}
        & h_0(\thetab) = \prod_{i=1}^d \frac{1}{ 2\pi I_0(\kappa_i)} ,
    \end{align*}
    which clearly satisfies \eqref{eq: condition} since the function is a constant, so the sine-skewed version of $f_{ivM}(\thetab)$ suffers from a singular FIM in the vicinity of symmetry.

    \item \textbf{Sine distribution:} The density is given by
    \begin{equation*}
        f_S(\thetab) = C_S \exp\left\{ \kappa_1 \cos(\theta_1) + \kappa_2 \cos(\theta_2) + \beta \sin(\theta_1)\sin(\theta_2) \right\}
    \end{equation*}
    with normalizing constant $C_S$ and concentration parameters $\kappa_1, \kappa_2 \geq0$, $\beta\in\R$. 
    It holds that $f_S(\thetab)$ can be written in the form \eqref{eq: solution} for $\gamma_1 = -\kappa_1$, $\gamma_2 = -\kappa_2$ and 
    \begin{equation*}
        h_0(\thetab) = C_S \exp\left\{  \beta \sin(\theta_1) \sin(\theta_2) \right\}.
    \end{equation*}
    In order to calculate $\alpha_1, \alpha_2$, we can set $t=1$ and solve $h_0(\thetab) = h_0(\thetab + \alphab)$ to obtain $\alpha_1 = \alpha_2 = 2\pi k$ for $k\in\Z \setminus \{0\}$.
    However, for $t=\frac{1}{4}$, $h_0(\thetab + \frac{1}{4} \alphab) = C_S \exp\left\{  \beta \cos(\theta_1) \cos(\theta_2) \right\} \neq h_0(\thetab)$.
    So, $h_0(\thetab)$ does not satisfy \eqref{eq: condition} and thus $f_S(\thetab)$ does not suffer from a singular FIM in the vicinity of symmetry.

    \item \textbf{Cosine distribution:} The density is given by \eqref{eq: cosine_density}.
    It holds that $f_C(\thetab)$ can be written in the form \eqref{eq: solution} for $\gamma_1 = -\kappa_1$, $\gamma_2 = -\kappa_2$ and 
    \begin{align*}
        & h_0(\thetab) = C_c \exp\left\{ \beta \cos(\theta_1 -\theta_2) \right\}.
    \end{align*}
    It is easy to see that for $\alpha_1 = \alpha_2 = 1$, $h_0(\thetab+\alphab t) = C_c \exp\left\{ \beta \cos(\theta_1 + t -\theta_2 - t) \right\} = h_0(\thetab)$ and thus $f_C(\theta_1, \theta_2)$ suffers from a singular FIM in the vicinity of symmetry.

    \item \textbf{multivariate von Mises distribution:}
Denoting $s(\thetab) = \left( \sin\theta_1, \ldots, \sin\theta_d \right)^\intercal$ and similarly $c(\thetab) = \left( \cos\theta_1, \ldots, \cos\theta_d \right)^\intercal$, \cite{mardia_multivariate_2008} proposes the following multivariate extensions of the Sine and Cosine models:
\begin{itemize}
    \item \textbf{multivariate extension of Sine distribution:} The density is given by
\begin{equation*}
    f_{mvS}(\thetab) = C_{mvS} \exp \left\{ \kappab^\intercal c(\thetab) + \frac{1}{2} s(\thetab)^\intercal \Lambda s(\thetab) \right\}
\end{equation*}
for normalizing constant $C_{mvS}$, $\kappab\in\R^d$ and symmetric matrix $\Lambda\in\R^{d\times d}$, 
with $(\Lambda)_{ii} = 0$ for $i\in\{1,\ldots,d\}$.
{It holds that $f_{mvS}(\thetab)$ can be written in the form \eqref{eq: solution} for $\gammab = -\kappab$ and 
    \begin{equation*}
        h_0(\thetab) = C_{mvS} \exp \left\{ \frac{1}{2} s(\thetab)^\intercal \Lambda s(\thetab) \right\}.
    \end{equation*}
Similar calculations as for the Sine model lead to the conclusion that $h_0(\thetab)$ does not satisfy \eqref{eq: condition}, and thus $f_{mvS}(\thetab)$ does not suffer from a singular FIM in the vicinity of symmetry.}
     \item \textbf{multivariate extension of Cosine distribution:} The density is given by
\begin{equation}\label{eq: mv_cosine}
    f_{mvC}(\thetab) = C_{mvC} \exp \Bigl( \kappab^\intercal c(\thetab) -  s(\thetab)^\intercal \Delta s(\thetab) - c(\thetab)^\intercal \Delta c(\thetab) \Bigr)
\end{equation}
for normalizing constant $C_{mvC}$, $\kappab =(\kappa_1, \ldots, \kappa_d)^\intercal \in \R^d$ and symmetric matrix $\Delta\in\R^{d\times d}$, $(\Delta)_{ij} = \delta_{ij} = \delta_{ji}$, $\delta_{ii} = 0$ for $i,j\in\{1,\ldots,d\}$. 
{Using the trigonometric identity 
\begin{equation*}
    \cos(\theta_i - \theta_j) = \cos(\theta_i) \cos(\theta_j) + \sin(\theta_i)\sin(\theta_j)
\end{equation*}
for $i\neq j$,} 
\eqref{eq: mv_cosine} can equivalently be written as
\begin{equation*}
    f_{mvC}(\thetab) = C_{mvC} \exp \left( 
    \sum_{i=1}^d \kappa_i \cos(\theta_i) - \sum_{j=1}^d \sum_{\substack{k=1 \\ k \neq j}}^d \delta_{jk} \cos(\theta_k - \theta_j) \right)
\end{equation*}
which can be written in the form \eqref{eq: solution} for $\gammab = -\kappab$ and
\begin{align*}
    h_0(\thetab) = C_{mvC} \exp \left( - \sum_{j=1}^d \sum_{ \substack{k=1 \\ k \neq j}}^d \delta_{jk} \cos(\theta_k - \theta_j) \right).
\end{align*}
\eqref{eq: condition} is satisfied for $\left( \alpha_1, \ldots, \alpha_d \right)^\intercal = \left( 1, \ldots, 1 \right)^\intercal$ and thus $f_{mvC}(\thetab)$ suffers from a singular FIM in the vicinity of symmetry.

\end{itemize}
\item \textbf{bivariate wrapped Cauchy distribution:}  The density is given by 
\begin{align*}
    f_{bwC} (\thetab) = c \{c_0 & - c_1\cos(\theta_1) - c_2\cos(\theta_2)  - c_3\cos(\theta_1)\cos(\theta_2) - c_4\sin(\theta_1)\sin(\theta_2) \}^{-1}
\end{align*}
for $c,c_0,c_1,c_2,c_3,c_4$ as given in \cite{kato_mobius_2015}. The density
$f_{bwC}(\thetab)$ can be written in the form \eqref{eq: solution} for some $\gammab \in \R^2$ and 
\begin{align*}
    h_0(\thetab) = f_{bwC} (\thetab) \exp \left( \sum^2_{i=1}\gamma_i \cos(\theta_i) \right).
\end{align*}
It can be shown that $h_0(\thetab)$ does not satisfy \eqref{eq: condition}, and thus $f_{bwC}(\thetab)$ does not suffer from a singular FIM in the vicinity of symmetry. The same reasoning applied to the trivariate wrapped Cauchy distribution of \cite{kato2025trivariatewrappedcauchycopula} leads to the same conclusion.

\end{itemize}
The conclusions for all models considered are as expected, and formally confirm the statement of \cite{symmetry_test} for the Sine and Cosine models.
Interestingly, we can observe some similarity between our results and the conclusions of \cite{Hallin_singularity2012}.
When considering a subclass of skewing functions $\Pi(\cdot)$ on the real line, they proved that only the Gaussian distribution admits a singular FIM in the vicinity of symmetry.
On the circle the equivalent of the Gaussian distribution is considered to be the von Mises, for whose extensions on the torus we prove that there are similar issues.

\section{Discussion} \label{sec: conclusion}

In this paper, we derive a general expression characterizing the class of symmetric models that suffer from the problem of singular FIM when combined with the sine-skewing mechanism.
As highlighted in the Introduction, singularity of the FIM is highly undesirable due to its negative impact on inferential procedures.

One approach addressed in the literature to tackle this problem involves the use of reparameterization techniques to eliminate such singularities. 
An important example is 
the parametrization introduced by \cite{HallinLey2014} based on the Gram–Schmidt orthogonalization in the space of scores. 
The core idea behind this method is to orthogonalize the score vector in order to make its components linearly independent.
However, it is often the case that reparameterization leads to a loss of interpretability, which is one of the original strengths of the sine skewing approach. 
Therefore, a different interesting path towards a solution would be to construct alternative skewing mechanisms that do not suffer from the same singularity issues. 

Some different skewing models have been proposed in the literature for circular distributions, see for example \cite{miyata2024}. 
However, no other skewing mechanism has been proposed, to the best of our knowledge, for the $d$-dimensional torus. 
Although the formulation of a skewing mechanism that does not suffer from singularity issues lies beyond the scope of this paper, it can be a direction for future research. \\

\appendix

\section{Proof of Theorem 1}\label{appendix}

{For the proof of the theorem we write the density as $f_0(\thetab)$,  dropping $\mub$ from the notation for convenience.
All results hold if we replace $\thetab$ by $\thetab - \mub$ (or equivalently $\theta_i$ by $\theta_i - \mu_i$ for $i=1,\ldots,d$).
}

\begin{proof}
The FIM $I_{f_0}$ for any symmetric distribution $f_0$ subject to sine-skewing is given by
\begin{equation*}
    I_{f_0} = \E \left[S_{f_0} S_{f_0}^\intercal \right],
\end{equation*}
with $S_{f_0}$ the notation for the score function evaluated in the vicinity of symmetry. The matrix
$I_{f_0}$ is singular if and only if the components of $S_{f_0}$ are linearly dependent.
The score function of \eqref{eq: sine_skewed_density} in the vicinity of symmetry corresponds to \eqref{eq: score}.
Therefore, the FIM is singular if and only if there exist real coefficients 
$\alpha_1,\ldots,\alpha_d,\beta_1,\ldots,\beta_d$ not all equal to 0 satisfying 
\begin{equation*} \label{eq: initial_form_pde}
    - \sum_{i=1}^d \alpha_i \frac{\frac{\partial}{\partial \theta_i}f_0(\thetab)}{f_0(\thetab)} + \sum_{i=1}^d \beta_i \sin(\theta_{i}) = 0,
\end{equation*}
which can be re-written as
\begin{equation}\label{eq: pde_to_be_solved}
    \sum_{i=1}^d \alpha_i \frac{\partial}{\partial \theta_i}f_0(\thetab) = 
    f_0(\thetab) \sum_{i=1}^d \beta_i \sin(\theta_{i}).
\end{equation}
Our goal is to solve the partial differential equation \eqref{eq: pde_to_be_solved} with the initial condition $\thetab=\thetab_0$, using the method of characteristics (for more details see \cite{pdes_book}). The transport structure of this PDE leads to a family of curves, known as the characteristic curves, along which the PDE \eqref{eq: pde_to_be_solved} reduces to an ODE. Those curves are classically defined by integrating the transport coefficients, and are here given by
\begin{align*}
\frac{d}{dt} \varphi(t,\thetab_{0}) = \alphab.
\end{align*}
Combining this with the initial condition $\varphi(0,\thetab_0)=\thetab_0$, the characteristic curves are the straight lines given by
\begin{align*}
    \varphi(t,\thetab_{0}) & = \thetab_{0} + t\alphab.
\end{align*}

\noindent Next, in order to simplify the PDE \eqref{eq: pde_to_be_solved} by absorbing its right-hand side, we define 
\begin{equation*}
   h_0(\thetab) = f_0(\thetab) \exp \left( \sum^d_{i=1} \gamma_i \cos (\theta_i)\right),  
\end{equation*}

\noindent where $\gamma_i = \frac{\beta_i}{\alpha_i}$ {for $\alpha_i \neq 0$}. 
Solving \eqref{eq: pde_to_be_solved} is equivalent to solving the homogeneous PDE
\begin{align}
    & \sum^d_{i=1} \alpha_i \frac{\partial h_0(\thetab)}{\partial \theta_i}\notag \\
    & = \sum^d_{i=1} \alpha_i \frac{\partial f_0(\thetab)}{\partial \theta_i}\cdot  \exp \left( \sum^d_{i=1} \gamma_i\cos (\theta_i)\right) - f_0(\thetab) \sum^d_{i=1} \beta_i \sin(\theta_i) \cdot  \exp \left( \sum^d_{i=1} \gamma_i\cos (\theta_i)\right)\notag\\
    & = 0.\label{eq: pde_to_be_solved h}
\end{align}
We then derive the simple ODE verified by the solution of this transport problem along the characteristic curves:
\begin{equation*}
     \frac{\mathrm{d} h_0(\varphi(t,\thetab_0))}{\mathrm{d} t } = \sum^d_{i=1}\frac{\partial h_0}{\partial\theta_i}(\varphi(t,\thetab_0))\cdot \frac{\partial (\thetab_{0} +t \alphab)}{\partial t}= \sum^d_{i=1} \alpha_i \frac{\partial h_0}{\partial \theta_i}(\varphi(t,\thetab_0)).
\end{equation*}
Reminding its equivalent formulation \eqref{eq: pde_to_be_solved h}, the PDE \eqref{eq: pde_to_be_solved} is true if and only if for all $t,\thetab_0$, $$\frac{\mathrm{d}h_0(\varphi(t,\thetab_0))}{\mathrm{d}t}=0,$$namely if $h_0$ is constant along each characteristic curve.\\ This holds if and only if there exist $\alpha_i \neq 0$ 
such that for all $t\in \mathbb{R}$ and $\thetab_0 \in [-\pi,\pi)^d$, 
    $h_0(\thetab_0 + t\alphab) = h_0(\thetab_0).$


\end{proof}

\noindent \textbf{Acknowledgements:} The authors would like to thank Christophe Ley for constructive discussions and valuable feedback.
\\

\noindent \textbf{Funding}:
Sophia Loizidou was supported by the grant PRIDE/21/16747448/MATHCODA from the Luxembourg National Research Fund.


 \bibliographystyle{elsarticle-harv} 
 \bibliography{references}

@article{mardia_multivariate_2008,
	title = {A {Multivariate} {Von} {Mises} {Distribution} with {Applications} to {Bioinformatics}},
	volume = {36},
	issn = {0319-5724},
	number = {1},
	urldate = {2023-12-14},
	journal = {The Canadian Journal of Statistics / La Revue Canadienne de Statistique},
	author = {Mardia, Kanti V. and Hughes, Gareth and Taylor, Charles C. and Singh, Harshinder},
	year = {2008},
	pages = {99--109},
}

@article{mardia_protein_2007,
	title = {{Protein Bioinformatics and Mixtures of Bivariate von {Mises} Distributions for Angular Data}},
	volume = {63},
	issn = {0006-341X},
	language = {eng},
	number = {2},
	journal = {Biometrics},
	author = {Mardia, Kanti V. and Taylor, Charles C. and Subramaniam, Ganesh K.},
	month = {jun},
	year = {2007},
	pmid = {17688502},
	pages = {505--512},
}

@article{singh_probabilistic_2002,
	title = {Probabilistic {Model} for {Two} {Dependent} {Circular} {Variables}},
	volume = {89},
	issn = {0006-3444},
	number = {3},
	urldate = {2023-06-19},
	journal = {Biometrika},
	author = {Singh, Harshinder and Hnizdo, Vladimir and Demchuk, Eugene},
	year = {2002},
	pages = {719--723},
}

@article{kato_mobius_2015,
	title = {A {M\"obius} transformation-induced distribution on the torus},
	volume = {102},
	journal = {Biometrika},
volume={102},
pages={359--370},
	author = {Kato, Shogo and Pewsey, Arthur},
	month = {may},
	year = {2015},
}

@article{abe2011symmetric,
  title={Symmetric circular models through duplication and cosine perturbation},
  author={Abe, Toshihiro and Pewsey, Arthur},
  journal={Computational Statistics \& Data Analysis},
  volume={55},
  number={12},
  pages={3271--3282},
  year={2011},
  publisher={Elsevier}
}

@article{pewsey_recent_2021,
	title = {Recent advances in directional statistics},
	volume = {30},
	issn = {1863-8260},
	language = {en},
	number = {1},
	urldate = {2023-11-13},
	journal = {TEST},
	author = {Pewsey, Arthur and García-Portugués, Eduardo},
	month = {mar},
	year = {2021},
	pages = {1--58},
}

@incollection{puerta_regularized_2015,
	address = {Cham},
	title = {Regularized {Multivariate} von {Mises} {Distribution}},
	volume = {9422},
	isbn = {978-3-319-24597-3 978-3-319-24598-0},
	language = {en},
	urldate = {2023-11-08},
	booktitle = {Advances in {Artificial} {Intelligence}},
	publisher = {Springer International Publishing},
	author = {Rodriguez-Lujan, Luis and Bielza, Concha and Larrañaga, Pedro},
	year = {2015},
	note = {Series Title: Lecture Notes in Computer Science},
	pages = {25--35},
}

@misc{nodehi_estimation_2018,
	title = {Estimation of {Multivariate} {Wrapped} {Models} for {Data} in {Torus}},
	urldate = {2023-11-08},
	publisher = {arXiv},
	author = {Nodehi, Anahita and Golalizadeh, Mousa and Maadooliat, Mehdi and Agostinelli, Claudio},
	month = {nov},
	year = {2018},
	note = {{arXiv}:1811.06007},
}

@article{mastrantonio_modelling_2022,
	title = {The {Modelling} of {Movement} of {Multiple} {Animals} that {Share} {Behavioural} {Features}},
	volume = {71},
	issn = {0035-9254},
	number = {4},
	urldate = {2023-11-08},
	journal = {Journal of the Royal Statistical Society Series C: Applied Statistics},
	author = {Mastrantonio, Gianluca},
	month = {aug},
	year = {2022},
	pages = {932--950},
}

@article{ley_simple_2014,
	title = {SIMPLE OPTIMAL TESTS FOR CIRCULAR REFLECTIVE SYMMETRY ABOUT A SPECIFIED MEDIAN DIRECTION.},
	volume = {24},
	issn = {1017-0405},
	number = {3},
	urldate = {2023-10-03},
	journal = {Statistica Sinica},
	author = {Ley, Christophe and Verdebout, Thomas},
	year = {2014},
	pages = {1319--1339},
}

@article{mardia_statistics_1975,
	title = {{Statistics of Directional Data}},
	volume = {37},
	journal = {Journal of the Royal Statistical Society: Series B (Methodological)},
	author = {Mardia, Kanti V},
	year = {1975},
	pages = {349--371},
}

@book{ley_modern_2017,
	address = {Boca Ratón, Florida},
	title = {Modern {Directional} {Statistics}},
	isbn = {978-1-4987-0664-3},
	publisher = {Chapman and Hall/CRC Press},
	author = {Ley, C. and Verdebout, T.},
	year = {2017},
}

@article{rivest_decentred_1997,
	title = {A decentred predictor for circular-circular regression},
	volume = {84},
	journal = {Biometrika},
	author = {Rivest, L.-P.},
	year = {1997},
	pages = {717--726},
}

@article{kato_distribution_2009,
	title = {A distribution for a pair of unit vectors generated by {Brownian} motion},
	volume = {15},
	journal = {Bernoulli},
	author = {Kato, Shogo},
	year = {2009},
	pages = {898--921},
}

@article{liu_phase_2006,
	title = {Phase analysis of circadian-related genes in two tissues},
	volume = {7},
	journal = {BMC Bioinformatics},
	author = {Liu, D. and Peddada, S. and Li, L. and Weinberg, C. R.},
	year = {2006},
	pages = {87},
}

@article{ameijeiras-alonso_sine-skewed_2022,
	title = {Sine-skewed toroidal distributions and their application in protein bioinformatics},
	volume = {23},
	issn = {1465-4644},
	number = {3},
	urldate = {2023-06-19},
	journal = {Biostatistics},
	author = {Ameijeiras-Alonso, Jose and Ley, Christophe},
	month = {jul},
	year = {2022},
	pages = {685--704},
}

@misc{kato2025trivariatewrappedcauchycopula,
      title={{The trivariate wrapped Cauchy copula}}, 
      author={Shogo Kato and Christophe Ley and Sophia Loizidou and Kanti V. Mardia},
      year={2025},
      eprint={2401.10824},
      archivePrefix={arXiv},
      primaryClass={stat.ME}, 
}

@article{Baba1981,
	author = {Baba, Y},
	journal = {Proceedings of the Institute of Statistical Mathematics},
	pages = {41--54},
	title = {Statistics of angular data: wrapped normal distribution model (in {Japanese})},
	volume = {28},
	year = {1981}}

@misc{symmetry_test,
      title={Construction of optimal tests for symmetry on the torus and their quantitative error bounds}, 
      author={Andreas Anastasiou and Christophe Ley and Sophia Loizidou},
      year={2025},
      eprint={2510.06055},
      archivePrefix={arXiv},
      primaryClass={math.ST},
}

@book{pdes_book,
author = {Novozhilov, Artem},
year = {2023},
month = {05},
pages = {},
title = {Undergraduate course in PDE (Spring 2023 version)}
}

@article{HallinLey2014,
  author    = {Marc Hallin and Christophe Ley},
  title     = {{Skew-symmetric distributions and Fisher information: The double sin of the skew-normal}},
  journal   = {Bernoulli},
  volume    = {20},
  number    = {3},
  pages     = {1432--1453},
  year      = {2014}
}

@article{miyata2024,
  title={An extension of sine-skewed circular distributions},
  author={Miyata, Yoichi and Shiohama, Takayuki and Abe, Toshihiro},
  journal={arXiv preprint arXiv:2402.09788},
  year={2024}
}

@incollection{bekker2022,
  title={{Generalized Skew-Symmetric Circular and Toroidal Distributions}},
  author={Bekker, Andriette and Nakhaei Rad, Najmeh and Arashi, Mohammad and Ley, Christophe},
  booktitle={Directional Statistics for Innovative Applications: A Bicentennial Tribute to Florence Nightingale},
  pages={161--186},
  year={2022},
  publisher={Springer}
}

@incollection{ameijeiras2022flexible,
  title={{Flexible Circular Modeling: A Case Study of Car Accidents}},
  author={Ameijeiras-Alonso, Jose and Crujeiras, Rosa M},
  booktitle={Directional Statistics for Innovative Applications: A Bicentennial Tribute to Florence Nightingale},
  pages={93--116},
  year={2022},
  publisher={Springer}
}

@article{rad2022enhancing,
  title={{Enhancing wind direction prediction of South Africa wind energy hotspots with Bayesian mixture modeling}},
  author={Rad, Najmeh Nakhaei and Bekker, Andriette and Arashi, Mohammad},
  journal={Scientific Reports},
  volume={12},
  number={1},
  pages={11442},
  year={2022},
  publisher={Nature Publishing Group UK London}
}

@article{HARVEY2024105450,
title = {Modelling circular time series},
journal = {Journal of Econometrics},
volume = {239},
number = {1},
pages = {105450},
year = {2024},
issn = {0304-4076},
author = {Andrew Harvey and Stan Hurn and Dario Palumbo and Stephen Thiele}
}

@article{Umbach2009,
title = {Building asymmetry into circular distributions},
journal = {Statistics \& Probability Letters},
volume = {79},
number = {5},
pages = {659-663},
year = {2009},
issn = {0167-7152},
author = {Dale Umbach and S. Rao Jammalamadaka}
}

@article{Azzalini1985,
  author  = {Azzalini, Adelchi},
  title   = {{A Class of Distributions Which Includes the Normal Ones}},
  journal = {Scandinavian Journal of Statistics},
  volume  = {12},
  pages   = {171--178},
  year    = {1985}
}

@article{Hallin_singularity2012,
author = {Marc Hallin and Christophe Ley},
title = {{Skew-symmetric distributions and Fisher information – a tale of two densities}},
volume = {18},
journal = {Bernoulli},
number = {3},
publisher = {Bernoulli Society for Mathematical Statistics and Probability},
pages = {747--763},
year = {2012},
}

@article{Rotnitzky2000,
author = {Andrea Rotnitzky and David R. Cox and Matteo Bottai and James Robins},
title = {{Likelihood-Based Inference with Singular Information Matrix}},
volume = {6},
journal = {Bernoulli},
number = {2},
publisher = {Bernoulli Society for Mathematical Statistics and Probability},
pages = {243 -- 284},
year = {2000},
}

@article{smeariness1,
author = {Shayan Hundrieser and Benjamin Eltzner and Stephan Huckemann},
title = {{Finite sample smeariness of Fréchet means with application to climate}},
volume = {18},
journal = {Electronic Journal of Statistics},
number = {2},
publisher = {Institute of Mathematical Statistics and Bernoulli Society},
pages = {3274 -- 3309},
year = {2024},
}

@article{smeariness2,
  title={Intrinsic means on the circle: uniqueness, locus and asymptotics},
  author={Hotz, Thomas and Huckemann, Stephan},
  journal={Annals of the Institute of Statistical Mathematics},
  volume={67},
  number={1},
  pages={177--193},
  year={2015},
  publisher={Springer}
}







\end{document}